\documentclass{amsart}
\usepackage{amssymb, bm, amsmath, latexsym, mathabx, dsfont,tikz}
\usepackage{multirow}
\usepackage{enumitem}

\renewcommand{\baselinestretch}{\baselinestretch}
\renewcommand{\baselinestretch}{1.1}
\numberwithin{equation}{section}

\newtheorem{thm}{Theorem}[section]
\newtheorem{lem}[thm]{Lemma}

\theoremstyle{definition}

\theoremstyle{remark}
\newtheorem{rmk}[thm]{Remark}

\numberwithin{equation}{section}

\newcommand{\ra}{{\rightarrow}}
\newcommand{\nra}{{\nrightarrow}}
\newcommand{\gen}{\text{gen}}

\newcommand{\z}{{\mathbb Z}}
\newcommand{\nn}{{\mathbb{N}_0}}
\newcommand{\q}{{\mathbb Q}}

\newcommand{\Mod}[1]{\ (\mathrm{mod}\ #1)}

\newcommand{\df}[1]{\langle #1 \rangle}

\newcommand{\ppp}[2][m]{{P_{#1,(#2)}}}
\newcommand{\pp}[3][m]{{P_{#1,(1^{#2},2^{#3})}}}

\newenvironment{newenum}
{\begin{enumerate}[label={\rm(\arabic*)}]}
	{\end{enumerate}}

\begin{document}


\author{Daejun Kim}
\address{Research Institute of Mathematics, Seoul National University,
 Seoul 08826, Korea}
\email{goodkdj@snu.ac.kr}




\thanks{This work was supported by the National Research Foundation of Korea(NRF) (NRF-2019R1A2C1086347) and by Basic Science Research Program through NRF funded by the Minister of Education (NRF-2020R1A6A3A03037816).}


\title[Weighted Sums of Generalized Polygonal numbers]{Weighted Sums of Generalized Polygonal Numbers with Coefficients $1$ or $2$}

\begin{abstract} 
	In this article, we consider weighted sums of generalized polygonal numbers with coefficients $1$ or $2$. We show that for any $m\ge10$, a weighted sum of generalized $m$-gonal numbers represents every non-negative integer if it represents $1$, $m-4$, and $m-2$. Furthermore, we study representations of sums of four generalized polygonal numbers with coefficients $1$ or $2$.
\end{abstract}

\maketitle

\section{Introduction}

For any positive integer $m\ge3$, the {\em $m$-gonal numbers} are the integers of the form
$$P_m(x) = (m-2) \cdot \left(\frac{x^2-x}{2} \right) + x, \text{ for }x \in \nn:=\mathbb{N}\cup \{0\}.$$
In 1638, Fermat claimed that every non-negative integer is written as the sum of $m$ $m$-gonal numbers, that is, there exists an $\bm{x}=(x_1,\ldots,x_m)\in\nn^m$ such that 
$$
\sum_{i=1}^m  P_m(x_i) = N
$$
for any $N\in\nn$.
Later, in 1770, Lagrange proved the four square theorem, which is exactly the case when $m=4$ of Fermat's assertion. In 1796, Gauss proved, so called, the Eureka Theorem, which is the case when $m=3$, and finally, Cauchy proved the general case $m\ge5$ in 1815. Nathanson (\cite{N1} and \cite[pp.3--33]{N2}) simplified and provided the proof of slightly stronger version of Cauchy's theorem. The Fermat's polygonal number theorem was generalized in many directions.

In 1830, Legendre refined the Fermat's polygonal number theorem and proved that any integer $N\ge 28(m-2)^3$ with $m\ge5$ is written as
$$
P_m(x_1)+P_m(x_2)+P_m(x_3)+P_m(x_4)+\delta_m(N),
$$
where $x_1,x_2,x_3,x_4\in\nn$, $\delta_m(N)=0$ if $m$ is odd, and $\delta_m(N)\in\{0,1\}$ if $m$ is even.
Nathanson \cite[p.33]{N2} simplified the proofs of Legendre's theorem.
Recently, Meng and Sun \cite{MS} strengthened Legendre's theorem by showing that if $m\equiv2\Mod{4}$ with $m\ge5$, then any integer $N\ge 28(m-2)^2$ can be written as the above with $\delta_m(N)=0$, while if $m\equiv 0\Mod{4}$ with $m\ge5$, there are infinitely many positive integers not of the form $P_m(x_1)+P_m(x_2)+P_m(x_3)+P_m(x_4)$ with $x_1,x_2,x_3,x_4\in\nn$.

On the other hand, Guy \cite{G} considered Fermat's polygonal number theorem for more general numbers $P_m(x)$ with $x\in\z$, which are called {\em generalized $m$-gonal numbers}. 
For a positive integer $m\ge 3$, $\bm{a}=(a_1,\ldots,a_k)\in\mathbb{N}^k$, and $\bm{x}=(x_1,\ldots,x_k)\in \z^k$, we define the sum 
\begin{equation}\label{polysum}
	P_{m,\bm{a}}(\bm{x}):=\sum_{i=1}^k a_iP_m(x_i).
\end{equation}
We say the sum $P_{m,\bm{a}}$ {\em represents} an integer $N$ if $P_{m,\bm{a}}(\bm{x})=N$ has an integer solution $\bm{x}\in\z^k$, and we write $N\ra P_{m,\bm{a}}$. The sum $P_{m,\bm{a}}$ is called {\em universal} if it represents every non-negative integer.
Guy \cite{G} asked for which $k\in\mathbb{N}$, the equation
$$
\sum_{i=1}^k P_m(x_i)=N
$$
has an integer solution $x_1,\ldots,x_k\in\z$ for any $N\in\nn$, that is, what is the minimal number $k_m$ such that the sum $P_{m,(1,\ldots,1)}$ ($1$ is repeated $k_m$ times) is universal. He explained that $k_m=3$ for $m\in\{3,5,6\}$ and $k_4=4$, and showed that $k_m\ge m-4$ for $m\ge8$, using the simple observation that the smallest generalized $m$-gonal number other than $0$ and $1$ is $m-3$.

Later, Sun \cite{S} proved that $\ppp[8]{1,1,1,1}$ is universal, which implies $k_8=4$, and also explained in the introduction that $k_7=4$. Indeed, note that $\ppp[7]{1,1,1}$ cannot represent $10$, and one may show that $\ppp[7]{1,1,1,1}$ is universal; thanks to the Legendre's theorem, one need only to check that any integers less than $3500=28(7-2)^3$ are represented by $\ppp[7]{1,1,1,1}$. In the same manner, one may verify that $k_9=5$.
Recently in \cite{BBKKPSSV}, it is shown that $k_m=m-4$ for $m\ge10$  (see the proof of Theorem \ref{mainthm} for another proof).
Therefore, the value $k_m$ is determined for any integer $m \ge 3$.

On the other hand, Kane and his collaborators \cite{BBKKPSSV} considered the specific case when
$$
\bm{a}=\bm{a}_{r,r-1,k}=(1,\ldots,1,r,\ldots,r),
$$
where $1$ is repeated $r-1$ times and $r$ is repeated $k-r+1$ times, and determined the minimal number $k$, say $k_{m,r,r-1}$, such that $P_{m,\bm{a}_{r,r-1,k}}$ is universal. In particular, they proved that $k_{m,2,1}=\lfloor \frac{m}{2} \rfloor$ for any $m\ge14$.

Motivated by this, in this article, we study the representations of the sum \eqref{polysum} with coefficients $1$ or $2$. For the sake of simplicity, for any non-negative integers $\alpha$ and $\beta$, we denote the vector
$$
(1^\alpha,2^\beta)=(\overbrace{1,\ldots,1}^{\alpha\text{-times}},\overbrace{2,\ldots,2}^{\beta\text{-times}}),
$$
where $1$ is repeated $\alpha$ times, and $2$ is repeated $\beta$ times.
The following theorem is the main result of this paper.

\begin{thm}\label{univcriterion}
	For any positive integer $m\ge10$, the sum $\pp{\alpha}{\beta}$ is universal if and only if it represents 
	$$
	1,\, m-4,\text{ and } m-2.
	$$
	Moreover, the sum $\pp{\alpha}{\beta}$ is universal if and only if it represents 
	$$
	1,3,5,10,19, \text{ and } 23 \text{ if } m=7, \text{ and}\quad 
	1,5,7, \text{ and } 34 \text{ if } m=9.
	$$
	
\end{thm}

Note that Theorem \ref{univcriterion} is complete in the sense that for each $m=3,4,5,6,$ and $8$, there is a criterion for determining the universality of an arbitrary sum $P_{m,\bm{a}}$  (see Remark \ref{rmk} \ref{rmk:1}). 
On the other hand, Theorem \ref{univcriterion} will be proved by using Lemma \ref{lemneccond} and Theorem \ref{mainthm}. 
When we prove Theorem \ref{mainthm}, Lemma \ref{lem1222(m-2)N} will be systematically applied for the case when $m\ge 19$, however, the same strategy does not work for $m\le 18$.
Moreover, neither Lemma \ref{lemneccond} nor Theorem \ref{mainthm} consider the case when $m=7$.
Therefore, in order to deal with the cases for those small positive integers, we need the following theorem, which is analogous to that of Legendre.

\begin{thm}\label{almostuniv}
	Let $m\ge 5$ and $N$ be integers. Let $\bm{a}$ be one of the vectors in 
	$$
	\{(1,1,1,1), (1,1,1,2), (1,1,2,2),(1,2,2,2)\},
	$$
	and put $C_{\bm{a}}=\frac{1}{8}, \frac{1}{10}, \frac{1}{3},$ and $\frac{7}{8}$ accordingly. Then we have the following:
	
	\begin{enumerate}[label={\rm(\arabic*)},leftmargin=*]
		\item\label{almostuniv:1} Every integer $N\ge C_{\bm{a}}(m-2)^3$ is represented by $P_{m,\bm{a}}$, unless 
		$$
		\bm{a}\in \{(1,1,1,1),(1,1,2,2)\} \text{ and } m\equiv 0\Mod{4} \text{ with } m>8.
		$$ 
		\item\label{almostuniv:2} In each exceptional case, there are infinitely many positive integers which are not represented by $P_{m,\bm{a}}$.

	\end{enumerate}
	
\end{thm}

\begin{rmk}\label{rmk}{\color{white} a}
	\begin{enumerate}[label={\rm(\arabic*)},leftmargin=*]
		\item\label{rmk:1}  In \cite{KL}, Kane and Liu showed that there exists a unique minimal positive integer $\gamma_m$ such that for any $\bm{a}\in \mathbb{N}^k$, $P_{m,\bm a}$ is universal if and only if it represents every $N \le \gamma_m$. 
		
		For the case when $3\le m\le 9$ with $m\not\in\{7,9\}$, the value $\gamma_m$ is known; $\gamma_3=\gamma_6=8$ by Bosma and Kane \cite{BK}, $\gamma_4=15$ by Conway-Schneeberger fifteen theorem \cite{B},\cite{C}, $\gamma_5=109$ by Ju \cite{J}, and $\gamma_8=60$ by Ju and Oh \cite{JO}, so those theorems give us criteria for $\pp{\alpha}{\beta}$ to be universal. 
		It seems to be difficult to obtain the values $\gamma_m$ for $m=7,9$.
		
		\item Generalizing the number $k_{m,r,r-1}$ defined and determined in \cite{BBKKPSSV} (introduced the above), for any $r\in\mathbb{N}$, let us define the number
		$$
		k_{m,r}:=\min \{k \mid P_{m,\bm{a}} \text{ is universal for some } \bm{a}\in \mathbb{N}_{\le r}^k\},
		$$
		where $\mathbb{N}_{\le r} = \{a \in\mathbb{N} \mid a\le r\}$. Then $k_{m,r}\le k_{m,r,r-1}$ follows from the definition. In particular, by Theorem \ref{mainthm}, $k_{m,2,1}=k_{m,2}=\lfloor\frac{m}{2}\rfloor$ for any odd integer $m$ with $m \ge 11$, while $k_{m,2}=\lfloor\frac{m}{2}\rfloor-1<\lfloor\frac{m}{2}\rfloor=k_{m,2,1}$ for any even integer $m$ with $m\ge 10$ and $k_{9,2}=4<5=k_{9,2,1}$.
		
		\item Theorem \ref{almostuniv} \ref{almostuniv:1} will be proved with the aid of Lemmas \ref{lemquatnec}-\ref{existb}.
		In those lemmas, the following system of diophantine equations are considered:
		$$
		\begin{cases}
			a_1x_1^2+a_2x_2^2+a_3x_3^2+a_4x_4^2=a\\
			a_1x_1+a_2x_2+a_3x_3+a_4x_4=b,
		\end{cases}	
		$$
		where $a,a_1,a_2,a_3,a_4\in\mathbb{N}$ and $b\in\z$.
		We study a solvability of the above equation over $\z$ by connecting it with an existence of a representation of a binary $\z$-lattice by a diagonal quaternary $\z$-lattice with certain constraint. 
		When $(a_1,a_2,a_3,a_4)=(1,1,1,1)$, the above equation was considered by Goldmakher and Pollack \cite{GP}, and our approach was made in this case by Hoffmann \cite{H}. Hence, our strategy could be considered as a generalization of the method used in \cite{H}.
		
		\item In addition to what we introduced previously, Meng and Sun \cite{MS} also showed that if $m\not\equiv0\Mod{4}$, then any $N\ge 1628(m-2)^3$ can be written as 
		$$
		P_m(x_1)+P_m(x_2)+2P_m(x_3)+2P_m(x_4) \text{ with } x_1,x_2,x_3,x_4\in\nn,
		$$
		while if $m\equiv 0\Mod{4}$, then there are infinitely many positive integers not of the above form. 
		Therefore, the statement
		
		\centerline{
			``any sufficiently large positive integer is represented by $P_{m,\bm{a}}$ over $\z$''} 
		\noindent has nothing to prove as we release the condition $x_i\in\nn$ to $x_i\in \z$, however, Theorem \ref{almostuniv} \ref{almostuniv:1} shows improvements on constants $C_{\bm{a}}$.
		On the other hand, Theorem \ref{almostuniv} \ref{almostuniv:2} tells us something more.
	\end{enumerate}
\end{rmk}

The rest of the paper is organized as follows. In Section \ref{preliminaries}, we introduce geometric language and theory of $\z$-lattices which are used to prove our theorems. In Section \ref{MainTheorem}, we classify all the universal sums $\pp{\alpha}{\beta}$ and prove Theorem \ref{univcriterion}.
Finally, in Section \ref{quaternarysum}, we prove Theorem \ref{almostuniv}, giving information for the integers represented by each of the sums $\ppp{1,1,1,1},\ppp{1,1,1,2},\ppp{1,1,2,2},$ and $\ppp{1,2,2,2}$.

\section{Preliminaries}\label{preliminaries}
In this section, we introduce several definitions, notations and well-known results on quadratic forms in the better adapted geometric language of quadratic spaces and lattices.
A $\z$-lattice $L=\z v_1+\z v_2+\dots+\z v_k$ of rank $k$ is a free $\z$-module equipped with non-degenerate symmetric bilinear form $B$ such that $B(v_i,v_j) \in \q$ for any $i,j$ with $1\le i, j \le k$. 
The corresponding quadratic map is defined by $Q(v)=B(v,v)$ for any $v \in L$.
We say a $\z$-lattice $L$ is {\it positive definite} if $Q(v)>0$ for any non-zero vector $v \in L$, and we say $L$ is {\it integral} if $B(v,w) \in \z$ for any $v,w \in L$.
Throughout this article, we always assume that a $\z$-lattice is positive definite and integral.
If $B(v_i,v_j)=0$ for any $i\neq j$, then we simply write 
$$L=\langle Q(v_1),\ldots,Q(v_k) \rangle.$$
The corresponding quadratic form in $k$ variables is defined by
$$
f_L(x_1,\ldots,x_k)=\sum_{1\le i,j \le k} B(v_i,v_j)x_ix_j.
$$

For two $\z$-lattices $\ell$ and $L$, we say $\ell$ is {\em represented} by $L$ if there is a linear map $\sigma : \ell \to L$ such that 
$$
B(\sigma(x),\sigma(y))=B(x,y) \quad \text{for any $x,y \in \ell$,} 
$$
and in this case, we write $\ell \ra L$. Such a linear map $\sigma$ is called a {\em representation} from $\ell$ to $L$.
If $\ell\ra L$ and $L\ra \ell$, then we say they are {\em isometric} to each other, and we write $\ell\cong L$.
For any prime $p$, we define localization of $L$ at $p$ by $L_p=L\otimes_\z \z_p$. 
We say $\ell$ is {\em locally represented} by $L$ if there is a local representation $\sigma_p : \ell_p \ra L_p$ which preserves the bilinear forms for any prime $p$. 
For a $\z$-lattice $L$, we define the {\em genus} $\gen(L)$ of $L$ as
$$
\gen(L)=\{K \text{ on } \q L \mid K_p \cong L_p \text{ for any prime } p\},
$$
where $\q L= \{ \alpha v \mid \alpha \in \q , v \in L\}$ is the quadratic space on which $L$ lies. The isometric relation induces an equivalence relation on $\gen(L)$, and we call the number of different equivalence classes in $\gen(L)$, the {\em  class number} of $L$.

Any unexplained notation and terminology can be found in \cite{OM2}.   

The following is well-known local-global principle for $\z$-lattices.

\begin{thm}\label{localglobal}
	Let $\ell$ and $L$ be $\z$-lattices. If $\ell$ is locally represented by $L$, then $\ell\ra L'$ for some $L'\in\gen(L)$.
	Moreover, if the class number of $L$ is one, then $\ell \ra L$  if and only if $\ell$ is locally represented by $L$.
\end{thm}
\begin{proof}
	See 102:5 of \cite{OM1}.
\end{proof}

Note that in case when $\ell$ is a unary $\z$-lattice $\df{n}$, $\ell \ra L$ if and only if $n=f_L(\bm{x})$ is solvable over $\z$, and $\ell$ is locally represented by $L$ if and only if $n=f_L(\bm{x})$ is solvable over $\z_p$ for any prime $p$.
The following lemma plays an important role in the proof of Theorem \ref{mainthm}, hence so does in the proof of Theorem \ref{univcriterion}.

\begin{lem}\label{lem1222(m-2)N}
	The sum $\ppp{1,2,2,2}$ represents every integer in the set
	$$\{ (m-2)N\in \nn : N\neq 2^{2s}(8t+1) \text{ for any } s,t\in\nn \}.$$
	
\end{lem}
\begin{proof}
	Consider $(x_1,\ldots,x_4)\in \z^4$ in the hyperplane $x_1+2x_2+2x_3+2x_4=0$. Then we have
	$$
	\begin{array}{rl}
		\ppp{1,2,2,2}(x_1,x_2,x_3,x_4)\!\!\!\!&
		=\frac{m-2}{2}((-2x_2-2x_3-2x_4)^2+2x_2^2+2x_3^2+2x_4^2)\\
		&=(m-2)(3x_2^2+3x_3^2+3x_4^2+4(x_2x_3+x_3x_4+x_4x_2)).
	\end{array}
	$$
	Note that the $\z$-lattice $L$ of rank $3$ to which the ternary quadratic form $3x_2^2+\cdots$ in the last equation corresponding has class number one. Moreover, one may check that $L$ locally represents every integer not of the form $2^{2s}(8t+1)$. Therefore, the lemma follows from Theorem \ref{localglobal}.	
\end{proof}
\section{Main Theorem}\label{MainTheorem}

\begin{lem}\label{lemneccond}
	Let $m\ge9$ be a positive integer and let $\alpha,\beta$ be non-negative integers. Assume that $\pp{\alpha}{\beta}$ is universal. Then we have the following:
	\begin{newenum}
		\item \label{lemneccond:1} $\pp{\alpha'}{\beta'}$ is universal for any integers $\alpha'\ge\alpha$ and $\beta'\ge\beta$,
		\item \label{lemneccond:2} $\pp{\alpha+2\beta'}{\beta-\beta'}$ is universal for any integer $0\le \beta'\le\beta$,
		\item \label{lemneccond:3} $\alpha\ge \max(m-2\beta-4,1)$,
		\item \label{lemneccond:4} if $\beta = \lfloor \frac{m}{2}\rfloor -2$, then $\alpha \ge 2$.
	\end{newenum}
\end{lem}
\begin{proof}
	The statements \ref{lemneccond:1} and \ref{lemneccond:2} are obvious. 
	On the other hand, since $\pp{\alpha}{\beta}$ represents $1$, we have $\alpha \ge 1$. Note that the smallest generalized $m$-gonal number other than $0$ and $1$ is $m-3$.
	So, in order for the equation
	$$
	m-4=\sum_{i=1}^\alpha P_m(x_i)+\sum_{i=\alpha+1}^{\alpha+\beta}2P_m(x_i)
	$$
	to have a solution $\bm{x}\in\z^{\alpha+\beta}$, we should have $\alpha+2\beta\ge m-4$. This proves \ref{lemneccond:3}.
	Now assume that $\beta=\lfloor \frac{m}{2}\rfloor -2$. Then we have $\alpha\ge1$ by \ref{lemneccond:3}.
	If the equation 
	$$
	m-2=P_m(x_1)+\sum_{i=2}^{1+\beta}2P_m(x_i)
	$$
	has a solution, then we should have $P_m(x_1)\in \{0,1,m-3\}$ and $2P_m(x_i)\in\{0,2\}$ for each $i$ with $2\le i\le 1+\beta$. However, it is impossible. Therefore, we should have $\alpha \ge 2$.
\end{proof}

\begin{thm}\label{mainthm}
	Let $m\ge 9$ be a positive integer and let $\alpha$ and $\beta$ be non-negative integers. Then the sum $\pp{\alpha}{\beta}$ is universal if and only if
	$$ \alpha\ge
	\begin{cases}
		1 & \text{if } \beta\ge \lfloor \frac{m}{2} \rfloor -1,\\
		2 & \text{if } \beta= \lfloor \frac{m}{2} \rfloor -2,\\
		m-2\beta-4 & \text{if } 0\le \beta \le \lfloor \frac{m}{2} \rfloor -3,
	\end{cases}
	$$
	unless $m=9$ and $\beta=3$, in which case $\pp[9]{\alpha}{3}$ is universal if and only if $\alpha\ge 2$.
\end{thm}

\begin{proof}
	The ``only if" part follows immediately from Lemma  \ref{lemneccond} \ref{lemneccond:3} and \ref{lemneccond:4}, and the fact that $\pp[9]{1}{3}$ cannot represent $34$.
	Now we prove the ``if" part. 
	Note that if we proved that $\pp{m-2\beta-4}{\beta}$ is universal when $\beta=\lfloor \frac{m}{2} \rfloor -3$, then Lemma \ref{lemneccond} \ref{lemneccond:2} implies that it is also universal for any $0\le \beta \le \lfloor \frac{m}{2} \rfloor -3$. 
	Moreover, when $m$ is even, if $\pp{2}{(m-6)/2}$ is universal, then so is  $\pp{2}{(m-4)/2}$ by Lemma \ref{lemneccond} \ref{lemneccond:1}.
	Therefore, in view of Lemma \ref{lemneccond} \ref{lemneccond:1}, it is enough to prove the following:\\
	\begin{newenum}
		\item[(i)] $\pp{1}{\lfloor m/2 \rfloor -1}$ for any $m\ge10$, $\pp[9]{2}{3}$, and $\pp[9]{1}{4}$ are universal,
		\item[(ii)] $\pp{2}{(m-5)/2}$ and $\pp{3}{(m-7)/2}$ are universal for any odd integer $m$,
		\item[(iii)] $\pp{2}{(m-6)/2}$ is universal for any even integer $m$.\\
	\end{newenum}
	
	First, we prove (i). The statement for any $m\ge14$ is proved in Theorem 1.1 (3) of \cite{BBKKPSSV} (see Section 4 of \cite{BBKKPSSV} for the proof). For any $9\le m \le 13$, note that $\lfloor \frac{m}{2} \rfloor -1\ge 3$. By Theorem \ref{almostuniv} (1), we know that $\pp{1}{3}$ represents every integer $N\ge\frac{7}{8}(m-2)^3$. Therefore, by checking  (by a computer program) whether or not the integers less than $\frac{7}{8}(m-2)^3$ are represented by $\pp{1}{3}$, one may determine the set $E(\pp{1}{3})$ of all integers that are not represented by $\pp{1}{3}$. From this set, one may conclude what we want; for example, we have $E(\pp[9]{1}{3})=\{34\}$, so $34$ is represented by both $\pp[9]{2}{3}$ and $\pp[9]{1}{4}$. Hence they are universal.
	
	Next, we prove (ii) and (iii). For any $9 \le m \le 18$, one may similarly prove that the sums are universal as above, by determining the set $E(\pp{1}{3}), E(\pp{2}{2}),$ or $E(\pp{3}{1})$ with the aid of Theorem \ref{almostuniv} (1).
	Now, we assume that $m\ge 19$. 
	We first prove the universality of $\pp{2}{(m-5)/2}=\pp{1}{3}+\pp{1}{(m-11)/2}$
	for any odd integer $m$ with $m\ge19$.
	Let $N$ be a non-negative integer and let
	$$
	R_1=\{0,1,\ldots,m-10,2m-11,3m-12,4m-13,4m-12,3m-9,2m-6,m-3\},
	$$
	$$
	R_2=\{r+2(m-2) \mid r\in R_1\} \quad \text{and} \quad R=R_1\cup R_2.
	$$
	Note that $R_i$ is a complete set of residues modulo $m-2$ for each $i=1,2$, and one may check that any integer $r\in R$ is represented by $\pp{1}{(m-11)/2}$.
	Also, one may check that every integer $N<6m-17$ is represented by $\pp{2}{(m-5)/2}$.
	Assume that $N\ge 6m-17$. For each $i=1,2$, there is a unique $r_i\in R_i$ such that 
	$$
	N\equiv r_i \Mod{m-2} \quad \text{ and } \quad N-r_i\ge0.
	$$
	Write $N-r_i=c_i(m-2)$.
	Since $r_2-r_1=2(m-2)$, we have $c_1-c_2=2$, hence for some $i_0\in\{1,2\}$, $c_{i_0}$ is not of the form $2^{2s}(8t+1)$ for any $s,t\in\nn$.
	Therefore, by Lemma \ref{lem1222(m-2)N}, $N-r_{i_0}$ is represented by $\pp{1}{3}$, hence $N=(N-r_{i_0})+r_{i_0}$ is represented by $\pp{2}{(m-5)/2}$.
	
	For the proof of the universality of  $\pp{3}{(m-7)/2}=\pp{1}{3}+\pp{2}{(m-13)/2}$ for any odd integer $m$ with $m\ge19$, and $\pp{2}{(m-6)/2}=\pp{1}{3}+\pp{1}{(m-12)/2}$ for any even integer $m$ with $m\ge19$, we take
	$$
	R_1=\{0,1,\ldots,m-11,2m-12,3m-13,4m-14,5m-15,4m-12,3m-9,2m-6,m-3\}.
	$$
	Then one may show the universality by repeating the same argument.
\end{proof}

We are now ready to prove Theorem \ref{univcriterion}.

\begin{proof}[Proof of Theorem \ref{univcriterion}]
	The proof is nothing but combining Lemma \ref{lemneccond} and Theorem \ref{mainthm} appropriately. 
	When $m\ge 10$, assume that $P=\pp{\alpha}{\beta}$ represents $1, m-4,$ and $m-2$.
	Since $1\ra P$, we should have $\alpha\ge1$. 
	Moreover, since $m-4 \ra P$, we should have $\alpha + 2\beta \ge m-4$ (see the proof of Lemma \ref{lemneccond}).
	Thus, by Theorem \ref{mainthm}, $P$ is universal unless $\beta=\lfloor\frac{m}{2}\rfloor-2$. 
	In the case when $\beta=\lfloor\frac{m}{2}\rfloor-2$, we should have $\alpha\ge2$ in order for $\pp{\alpha}{\lfloor m/2 \rfloor-2}$ to represent $m-2$ (see the proof of Lemma \ref{lemneccond}), and therefore, $\pp{\alpha}{\lfloor m/2 \rfloor-2}$ is universal by Theorem \ref{mainthm}.
	
	When $m=9$, one may similarly show that if $P=\pp[9]{\alpha}{\beta}$ represents $1,5,$ and $7$, then it is universal, except for $\ppp[9]{1,2,2,2}$. Using Theorem \ref{almostuniv}, we may verify that $E(\ppp[9]{1,2,2,2})=\{34\}$, and so both $\pp[9]{2}{3}$ and $\pp[9]{1}{4}$ are universal. Therefore, we may conclude that if $P$ represents $1,5,7,$ and $34$, then it is universal.
	
	When $m=7$, one may show that if $P=\pp[7]{\alpha}{\beta}$ represents $1,3,$ and $5$ then $P$ should contain $\ppp[7]{1,1,1},\ppp[7]{1,1,2},$ or $\ppp[7]{1,2,2}$, and they don't represent $10,23,$ or $19$, respectively.
	On the other hand, using Theorem \ref{almostuniv}, we may verify that each of the sums $\ppp[7]{1,1,1,1},\ppp[7]{1,1,1,2},\ppp[7]{1,1,2,2},$ and $\ppp[7]{1,2,2,2}$ is universal. Therefore, we may conclude that if $P$ represents $1,3,5,10,19,$ and $23$, then it is universal.
\end{proof}

\section{Representations of quaternary sums $\pp{\alpha}{\beta}$}\label{quaternarysum}

In this section, we prove Theorem \ref{almostuniv}.
Throughout this section, let us set several notations.
For each $\bm{a}=(a_1,a_2,a_3,a_4)\in \mathbb{N}^4$, we put $A=A_{\bm{a}}=\sum_{i=1}^4a_i$, and we define the quaternary diagonal $\z$-lattice $L_{\bm{a}}$ with basis $\{w_1,w_2,w_3,w_4\}$ by
$$
L_{\bm{a}}=\z w_1+\z w_2 + \z w_3 + \z w_4 =\df{a_1,a_2,a_3,a_4}.
$$
Let 
$$
S:=\{(1,1,1,1), (1,1,1,2), (1,1,2,2),(1,2,2,2)\},
$$
and for each $\bm{a}\in S$, we define the set of integers
$$
E_{\bm{a}}=
\begin{cases}
	\{2^{2s}(8t+7) \mid s\in\nn, t \in \z \} & \text{if } \bm{a}=(1,1,1,1) \text{ or } (1,1,2,2),\\
	\{5^{2s+2}(5t\pm 2)\mid s\in\nn, t \in \z \} & \text{if } \bm{a}=(1,1,1,2),\\
	\{2^{2s}(16t+14)\mid s\in\nn, t \in \z \} & \text{if } \bm{a}=(1,2,2,2).
\end{cases}
$$
For a binary $\z$-lattice $\ell=\z v_1 + \z v_2$, we simply write $\ell=[Q(v_1),B(v_1,v_2),Q(v_2)]$.
The following lemmas will play crucial roles in proving Theorem \ref{almostuniv} (1).

\begin{lem}\label{lemquatnec}
	Let $\bm{a}=(a_1,a_2,a_3,a_4)\in\mathbb{N}^4$, $a\in \mathbb{N}$, and $b\in \z$. Assume that the following system of diophantine equations
	\begin{equation}\label{eq1}
		\begin{cases}
			a_1x_1^2+a_2x_2^2+a_3x_3^2+a_4x_4^2=a\\
			a_1x_1+a_2x_2+a_3x_3+a_4x_4=b
		\end{cases}	
	\end{equation}
	has an integer solution $x_1,x_2,x_3,x_4\in\z$. Then we have
	\begin{newenum}
		\item $a\equiv b\Mod{2}$ and $Aa-b^2\ge0$,
		\item the integer $N:=\frac{m-2}{2}(a-b)+b$ is represented by $\ppp{a_1,a_2,a_3,a_4}$.
	\end{newenum}
\end{lem}
\begin{proof}
	Since $x_i^2\equiv x_i\Mod{2}$, we necessarily have $a\equiv b\Mod{2}$, and the inequality $Aa-b^2\ge0$ is nothing but Cauchy-Schwarz inequality. Moreover, note that
	$$
	\frac{m-2}{2}(a-b)+b
	=\sum_{i=1}^4 a_i\left(\frac{m-2}{2}(x_i^2-x_i)+x_i\right)
	=\ppp{a_1,a_2,a_3,a_4}(x_1,x_2,x_3,x_4).
	$$
	This proves the lemma.
\end{proof}

\begin{lem}\label{equivlem}
	Let $\bm{a}\in S$, and let $a$ and $b$ be integers such that 
	$$a\equiv b\Mod{2} \quad  \text{and} \quad Aa-b^2>0.$$
	Then the following are equivalent. 
	\begin{newenum}
		\item\label{equivlem:1} The equation \eqref{eq1} has an integer solution $x_1,x_2,x_3,x_4\in\z$.
		\item\label{equivlem:2} There exist a representation $\sigma : [A,b,a]\ra L_{\bm{a}}$ such that 
		$$\sigma(v_1)=w_1+w_2+w_3+w_4.$$
		\item\label{equivlem:3} The binary $\z$-lattice $[A,b,a]$ is represented by the quaternary $\z$-lattice $L_{\bm{a}}$.
		\item\label{equivlem:4} The positive integer $Aa-b^2$ is not contained in $E_{\bm{a}}$.
	\end{newenum}
\end{lem}
\begin{proof}
	We first prove $\ref{equivlem:3}\Leftrightarrow\ref{equivlem:4}$. 
	Note that the class number of $L_{\bm{a}}$ is one for any $\bm{a}\in S$.
	Therefore, by Theorem \ref{localglobal} $[A,b,a]$ is represented by $L_{\bm{a}}$ if and only if $[A,b,a]$ is locally represented by $L_{\bm{a}}$. By Theorem 1 and 3 of \cite{OM1}, one may check, under the assumption on $a$ and $b$, that $[A,b,a]$ is locally represented by $L_{\bm{a}}$ if and only if $Aa-b^2\not\in E_{\bm{a}}$.
	
	Next, we prove $\ref{equivlem:1}\Leftrightarrow\ref{equivlem:2}$.
	Assume that there exist integers $x_1,x_2,x_3,x_4\in\z$ satisfying \eqref{eq1}. Define a linear map $\sigma : [A,b,a]\ra L_{\bm{a}}$ by 
	$$
	\sigma(v_1)=w_1+w_2+w_3+w_4 \quad \text{and}\quad \sigma(v_2)=\sum_{i=1}^4 x_iw_i.
	$$
	Then $\sigma : [A,b,a] \ra L$ is a representation since we have 
	\begin{equation*}
		\begin{cases}
			Q(\sigma(v_1))=A=Q(v_1),\\
			Q(\sigma(v_2))=a_1x_1^2+a_2x_2^2+a_3x_3^2+a_4x_4^2=a=Q(v_2),\\
			B(\sigma(v_1),\sigma(v_2))=a_1x_1+a_2x_2+a_3x_3+a_4x_4=b=B(v_1,v_2),
		\end{cases}
	\end{equation*}
	from \eqref{eq1}. This proves $\ref{equivlem:1} \Rightarrow \ref{equivlem:2}$, and $\ref{equivlem:2} \Rightarrow \ref{equivlem:1}$ can also be easily proved.
	
	Finally, we prove $\ref{equivlem:2}\Leftrightarrow\ref{equivlem:3}$. We need only to prove $\ref{equivlem:3}\Rightarrow\ref{equivlem:2}$. 
	Assume that there is a representation $\tau : [A,b,a] \ra L_{\bm{a}}$. 
	By changing the sign of $w_i$ for $1\le i \le 4$ or by interchanging $w_i$ and $w_j$ for $1\le i,j \le 4$ with $a_i=a_j$ if necessarily, we may assume that either $\tau(v_1)=w_1+w_2+w_3+w_4$ or 
	$$\tau(v_1)=
	\begin{cases}
		2w_1& \text{if } \bm{a}=(1,1,1,1),\\
		2w_1+w_2& \text{if } \bm{a}=(1,1,1,2),\\
		2w_1+w_3& \text{if } \bm{a}=(1,1,2,2).
	\end{cases}
	$$
	In the former case, we are done by taking $\sigma=\tau$. 
	To deal with the latter case, let $\tau(v_2)=\sum_{i=1}^4 y_iw_i$ ($y_i\in \z$). 
	
	First, we consider the case when $\bm{a}=(1,1,1,2)$ and $\tau(v_1)=2w_1+w_2$.
	Consider the $\q$-linear map $\sigma_T$ from $\q L_{\bm{a}}$ to itself defined by 
	$$
	\sigma_T(w_j)=\sum_{i=1}^4 t_{ij}w_i \text{ for each } 1\le j \le 4, \text{ where $T=(t_{ij})=\frac{1}{2}\cdot{\tiny \setlength\arraycolsep{2pt} \begin{pmatrix}
				0 & 2 & 0 & 0 \\
				1 & 0 & 1 & 2\\
				1 & 0 & 1 &-2\\
				1 & 0 &-1 & 0
		\end{pmatrix}}$}.
	$$
	Then $\sigma_{T}\in O(\q L_{\bm{a}})$. If we let $\sigma=\sigma_{T}\circ\tau$, then
	$$
	\sigma(v_1)=\sigma_{T}(2w_1+w_2)=w_1+w_2+w_3+w_4.
	$$
	On the other hand, since $\tau:[A,b,a]\ra L_{\bm{a}}$ is a representation, we have
	$$
	y_1^2 + y_2^2 + y_3^2 + 2y_4^2 = a \quad \text{ and } \quad 2y_1+y_2=b.
	$$
	Note that since $y_2^2 \equiv y_2\equiv b \equiv a \Mod{2}$, we have $y_1 \equiv y_1^2\equiv y_3^2 \equiv y_3 \Mod{2}$. 
	Therefore, $\sigma(v_2)=\sigma_{T}(\sum_{i=1}^4 y_iw_i) =:\sum_{i=1}^4 x_iw_i\in L_{\bm{a}}$, since 
	$$
	(x_1,x_2,x_3,x_4)=\left(y_2, \frac{y_1+y_3}{2}+y_4 , \frac{y_1+y_3}{2}-y_4, \frac{y_1-y_3}{2}\right)\in \z^4,
	$$
	which implies that $\sigma:[A,b,a]\ra L_{\bm{a}}$ is a representation that we want to find.
	
	For each of the remaining two cases, one may follow the argument similar to the above to show that $\sigma=\sigma_T\circ \tau$ is a representation that we desired, by taking 
	$$
	T=\frac{1}{2}\cdot{\tiny \setlength\arraycolsep{2pt} \begin{pmatrix}
			1 & 1 & 1 & 1 \\
			1 & 1 & -1 & -1\\
			1 & -1 & 1 &-1\\
			1 & -1 & -1 & 1
	\end{pmatrix}} 
	\quad \text{or} \quad 
	\frac{1}{2}\cdot{\tiny \setlength\arraycolsep{2pt} \begin{pmatrix}
			0 & 0 & 2 & 2 \\
			0 & 0 & 2 & -2\\
			1 & -1 & 0 & 0\\
			1 & 1 & 0 & 0
	\end{pmatrix}},
	$$
	according as $(\bm{a},\tau(v_1))=((1,1,1,1),2w_1)$ or $((1,1,2,2),2w_1+w_3)$.
\end{proof}

\begin{lem}\label{existb}
	Let $\bm{a}\in S$ and put $B_{\bm{a}}=2,2,4,$ and $7$ according as 
	$$
	\bm{a}=(1,1,1,1),(1,1,1,2),(1,1,2,2),\text{ and }(1,2,2,2).
	$$	
	Let $m\ge 5$ be an integer and let $I$ be a closed interval whose length is longer than or equal to $B_{\bm{a}}(m-2)$. 
	Then for any integer $N$, there exists an integer $b \in I$ such that
	\begin{equation}\label{condb}
		N\equiv b \Mod{m-2}\quad  \text{and} \quad Aa-b^2\not\in E_{\bm{a}},
	\end{equation}
	where $a=2\left( \frac{N-b}{m-2}\right)+b$, unless $m\equiv0\Mod{4}$ and $\bm{a}\in\{(1,1,1,1),(1,1,2,2)\}$.
\end{lem}

\begin{proof}
	For any integer $N$, let $b_0$ be the smallest integer in the interval $I$ such that $N\equiv b_0 \Mod{m-2}$. For an integer $k$, we define 
	$$
	b_k=b_0+k(m-2), \, a_k=2\left( \frac{N-b_k}{m-2}\right)+b_k,\text{ and } D_k=Aa_k-b_k^2.
	$$
	Note that $a_k=a_0+k(m-4)\in \z$ for any integer $k$. 
	We will show that 
	$$
	D_k\not\in E_{\bm{a}} \text{ for some } 0\le k \le B_{\bm{a}}-1.
	$$
	Then the lemma follows since $b=b_k$ satisfies \eqref{condb} and the interval $I$ contains $B_{\bm{a}}(m-2)$ consecutive integers.\\
	
	\noindent {\bf (Case 1)} $\bm{a}=(1,1,1,1)$ and $m\not\equiv 0\Mod{4}$ ($E_{\bm{a}}=\{2^{2s}(8t+7) \mid s\in\nn, t \in \z \}$).
	
	If $N\not\equiv 0\Mod{2}$ or $m\not\equiv0\Mod{2}$, then one may note that $b_k$ is an odd integer for some $k\in \{0,1\}$. Then $D_k=4a_k-b_k^2\equiv3\Mod{8}$, since $a_k\equiv b_k\equiv 1\Mod{2}$. Hence we have $D_k\not\in E_{\bm{a}}$. 
	
	Otherwise, $N\equiv 0\Mod{2}$ and $m\equiv2\Mod{4}$. Thus, $b_i\equiv a_i\equiv 0\Mod{2}$ for any integer $i$. Moreover, since $m-4\equiv 2\Mod{4}$, $a_k \equiv 2 \Mod{4}$ for some $k\in\{0,1\}$. Since $D_k\equiv 4$ or $8\Mod{16}$, we have $D_k\not\in E_{\bm{a}}$.\\
	
	\noindent {\bf (Case 2)} $\bm{a}=(1,1,1,2)$ ($E_{\bm{a}}=\{5^{2s+2}(5t\pm2)\mid s\in\nn, t \in \z \}$).
	
	If $(m-2) \not\equiv 0 \Mod{5}$, then $b_k\not\equiv0\Mod{5}$ for some $k\in \{0,1\}$.
	Note that $D_k=5a_k-b_k^2\equiv \pm 1 \Mod{5}$. Hence, $D_k\not\in E_{\bm{a}}$.
	
	Now assume that $(m-2) \equiv 0 \Mod{5}$. Note that $N\equiv b_0\Mod{5}$. If $N\not\equiv0\Mod{5}$, then $D_0\equiv\pm 1\Mod{5}$, hence $D_0\not\in E_{\bm{a}}$. If $5\mid N$, then $a_k\not\equiv 0\Mod{5}$ for some $k\in\{0,1\}$. Since $b_k\equiv b_0\equiv0\Mod{5}$, we have $5|D_k$ but $25\nmid D_k$, hence $D_k\not\in E_{\bm{a}}$.\\
	
	\noindent {\bf (Case 3)} $\bm{a}=(1,1,2,2)$ and $m\not\equiv 0\Mod{4}$ ($E_{\bm{a}}=\{2^{2s}(8t+7)\mid s\in\nn, t \in \z \}$).

	If $N\not\equiv 0\Mod{2}$ or $m\not\equiv0\Mod{2}$, then one may note that $b_k$ is an odd integer for some $k\in \{0,1\}$.  Then $D_k=6a_k-b_k^2\equiv 1 \text{ or } 5 \Mod{8}$, since $a_k\equiv b_k\equiv 1\Mod{2}$. Hence we have $D_k\not\in E_{\bm{a}}$. 
	
	Otherwise, $N\equiv 0\Mod{2}$ and $m\equiv2\Mod{4}$. So, $b_i\equiv a_i\equiv 0\Mod{2}$, hence $D_i\equiv 0\Mod{4}$  for any integer $i$.
	Note that $D_{i_1}\equiv D_{i_2}\Mod{16}$ if and only if 
	$$
	4(i_1-i_2)\left(3\left(\frac{m-4}{2}\right)-b_0\left(\frac{m-2}{2}\right)+(i_1+i_2)\left(\frac{m-2}{2}\right)^2\right) \equiv 0\Mod{16}.
	$$
	Since $m-4\equiv 2\Mod{4}$ and $m-2\equiv0\Mod{4}$, it is equivalent to $i_1\equiv i_2\Mod{4}$. Hence, we have
	$$
	\{D_i \text{ mod } 16 \mid i=0,1,2,3\}=\{0,4,8,12\}.
	$$
	Therefore, $D_k\not\in E_{\bm{a}}$ for some $k\in\{0,1,2,3\}$.\\
	
	\noindent {\bf (Case 4)} $\bm{a}=(1,2,2,2)$ ($E_{\bm{a}}=\{2^{2s}(16t+14)\mid s\in\nn, t \in \z \}$).

	We will show that $D_k\not\in E_{\bm{a}}$ for some integer $k$ with $0\le k \le 6$. 
	We may assume that $D_0 \in E_{\bm{a}}$, since otherwise we are done.
	For any integer $i$, define 
	$$
	\Delta_i=D_i-D_0=7(m-4)i-(b_0+i(m-2))^2+b_0^2.
	$$
	{\bf (4-1)} Assume that $m \equiv 1 \Mod{2}$. Note that $\Delta_i\equiv 0\Mod{2}$ for any $i$. Moreover, for integers $i_1,i_2$ with $i_1\equiv i_2\Mod{2}$, we have $\Delta_{i_1}\equiv\Delta_{i_2}\Mod{8}$ if and only if
	$$
	\left(\dfrac{i_1-i_2}{2}\right)(7(m-4)-(m-2)((i_1+i_2)(m-2)+2b_0))\equiv 0 \Mod{4}.
	$$
	Since both $m-4$ and $m-2$ are odd integers, it is equivalent to $i_1 \equiv i_2 \Mod{8}$. Hence, $\{\Delta_i \Mod{8} : i=0,2,4,6\}=\{0,2,4,6\}$, so 
	$\Delta_k\equiv 4 \Mod{8}$ for some $k\in\{0,2,4,6\}$.
	Therefore, one may show that $D_k=D_0+\Delta_k\not\in E_{\bm{a}}$.\\
	
	\noindent{\bf (4-2)} Assume that $m \equiv 2 \Mod{4}$. In this case, one may easily show that $\Delta_2\equiv 4 \Mod{8}$. Hence, $D_2=D_0+\Delta_2\not \in E_{\bm{a}}$.\\
	
	\noindent{\bf (4-3)} Assume that $m \equiv 4 \Mod{8}$. If $b_0\equiv 0 \Mod{2}$, then $\Delta_1 \equiv 4 \Mod{8}$, so $D_1=D_0+\Delta_1\not \in E_{\bm{a}}$. 
	Now assume that $b_0\equiv 1\Mod{2}$. Then $\Delta_i\equiv0\Mod{8}$ for any $i$. 
	Moreover, for integers $i_1,i_2$ with $i_1\equiv i_2\Mod{2}$, one may show that 
	$$
	\Delta_{i_1}\equiv \Delta_{i_2} \Mod{32} \Leftrightarrow i_1\equiv i_2 \Mod{8}.
	$$
	Hence, $\{\Delta_i \Mod{32} : i=0,2,4,6\}=\{0,8,16,24\}$. 
	
	If $D_0$ is of the form $4(16t+14)$, then $D_k=D_0+\Delta_k\not\in E_{\bm{a}}$ for some $k\in\{0,2,4,6\}$ with $\Delta_k\equiv 16 \Mod{32}$.
	Otherwise, for some $k\in\{0,2,4,6\}$ with $\Delta_k\equiv 8 \Mod{32}$, $D_k=D_0+\Delta_k\not\in E_{\bm{a}}$.\\
	
	\noindent{\bf (4-4)} The proof of the case when $m \equiv 0 \Mod{8}$ is quite similar to that of (4-3).
\end{proof}

We are now ready to prove Theorem \ref{almostuniv} \ref{almostuniv:1}.
\begin{proof}[Proof of Theorem \ref{almostuniv} {\rm (1)}]
	For each $\bm{a}\in S$, let $A=A_{\bm{a}}$ and $B=B_{\bm{a}}$, and let 
	$$
	I=I_{\bm{a}}=\left[\frac{A}{2}\left(\frac{m-4}{m-2}\right)- \frac{B}{2}(m-2),\frac{A}{2}\left(\frac{m-4}{m-2}\right)+ \frac{B}{2}(m-2)\right]
	$$ 
	be a closed interval whose length is $B(m-2)$, and let $N\ge C_{\bm{a}}(m-2)^3$ be an integer. Then by Lemma \ref{existb}, there exists an integer $b\in I$ such that
	$$
	N\equiv b \Mod{m-2}\quad  \text{and} \quad Aa-b^2\not\in E_{\bm{a}},
	$$
	where $a=2\left( \frac{N-b}{m-2}\right)+b$. 
	Note that $a\equiv b\Mod{2}$ and since $C_{\bm{a}}=\frac{B^2}{8A}$, we have
	$$
	\displaystyle\max_{b\in I} \left[\left(\frac{m-2}{2A}\right)b^2-\left(\frac{m-4}{2}\right)b\right]
	=
	\frac{B^2}{8A}(m-2)^3-\frac{A(m-4)^2}{8(m-2)}
	<
	C_{\bm{a}}(m-2)^3
	$$
	for any $m\ge 5$.
	Since $Aa-b^2>0$ if and only if $N>\left(\frac{m-2}{2A}\right)b^2-\left(\frac{m-4}{2}\right)b$, we have $Aa-b^2>0$. 
	Therefore, by Lemma \ref{equivlem}, there are integers $x_1,\ldots,x_4$ such that
	$$
	a_1x_1^2+a_2x_2^2+a_3x_3^2+a_4x_4^2=a \quad \text{and} \quad a_1x_1+a_2x_2+a_3x_3+a_4x_4=b.
	$$
	Therefore, by Lemma \ref{lemquatnec} $N=\frac{m-2}{2}(a-b)+b$ is represented by $P_{m,\bm{a}}$.
\end{proof}

Now, we are ready to prove Theorem \ref{almostuniv} \ref{almostuniv:2}.

\begin{proof}[Proof of Theorem \ref{almostuniv} {\rm (2)}]
	Let $\bm{a}$ be either $(1,1,1,1)$ or $(1,1,2,2)$, $A=A_{\bm{a}}$ and let $m=4l+4$ for some integer $l\ge 2$. 
	Let $N_0$ be a positive integer such that 
	$$
	N_0\nra P_{m,\bm{a}} \quad \text{and} \quad (2l+1)N_0+Al^2 \equiv 0 \Mod{4}.
	$$
	Note that such an integer exists; for example, one may take $N_0=10$ when $\bm{a}=(1,1,2,2)$ and $l\ge 2$ is odd, and $N_0=8$ otherwise.
	Moreover, we put 
	$$
	n=\text{ord}_{(2l+1)}(2)=\text{ord}_{(2l+1)}(l+1),
	$$
	where $\text{ord}_{b}(a)$ denotes the smallest positive integer $k$ such that $a^k \equiv 1 \pmod{b}$ for any positive integers $a$ and $b$ with $\text{gcd}(a,b)=1$. 
	
	We claim that for any $t\in\nn$, the integer
	$$
	N_t=N_{t,\bm{a}}:=\frac{4^{nt}((2l+1)N_0+Al^2)-Al^2}{2l+1}
	$$
	is not represented by $P_{m,\bm{a}}$. 
	Since $N_t\in\mathbb{N}$, the theorem follows directly from this claim.
	We will show that for $t\in\mathbb{N}$, $N_t\ra P_{m,\bm{a}}$ implies $N_{t-1}\ra P_{m,\bm{a}}$. Then since $N_0$ is not represented by $P_{m,\bm{a}}$, the claim follows. 
	Note that for any integer $N$, we have
	$$ 
	N=P_{m,\bm{a}}(x_1,x_2,x_3,x_4) \quad \Leftrightarrow \quad (2l+1)N+Al^2=\sum_{i=1}^{4}a_i\left((2l+1)x_i -l \right)^2.
	$$
	Assume that $N_t=P_{m,\bm{a}}(x_1,x_2,x_3,x_4)$ for some $x_1,x_2,x_3,x_4\in\z$.
	Then we have
	$$
	4^{n}((2l+1)N_{t-1}+Al^2)=4^{nt}((2l+1)N_0+Al^2) = \sum_{i=1}^{4}a_i\left((2l+1)x_i -l \right)^2\text{.}
	$$
	Since the left hand side is a multiple of $16$, we have $(2l+1)x_i-l \equiv 0 \Mod{2}$ for any $1\le i \le 4$. Since $\left( (2l+1)x_i-l \right) /2 \equiv -l(l+1) \pmod{2l+1}$, there exist integers $y_1,y_2,y_3,y_4\in\z$ such that 
	$$
	4^{n-1}((2l+1)N_{t-1}+Al^2) = \sum_{i=1}^{4}a_i \left((2l+1)y_i -l(l+1) \right)^2\text{.}
	$$
	Applying similar arguments recursively, we have 
	$$
	(2l+1)N_{t-1}+Al^2 = \sum_{i=1}^{4} a_i\left((2l+1)z_i -l(l+1)^n \right)^2
	$$
	for some integers $z_1,z_2,z_3,z_4\in\z$.
	Since $(2l+1)z_i -l(l+1)^n\equiv -l \Mod{2l+1}$,
	$$
	(2l+1)N_{t-1}+Al^2=\sum_{i=1}^{4} a_i \left((2l+1)z'_i -l \right)^2\text{.}
	$$
	for some integers $z_1',z_2',z_3',z_4'\in\z$, and so $N_{t-1}\ra P_{m,\bm{a}}$.
	This proves the claim, hence so does the theorem.
\end{proof}

\end{document}